\documentclass[11pt,a4paper]{scrartcl}

\usepackage[margin=1in]{geometry}

\usepackage[T1]{fontenc}
\usepackage[utf8]{inputenc}
\usepackage{lmodern}
\usepackage{amssymb}
\usepackage{amsmath}
\usepackage{amsfonts}
\usepackage{amsthm}
\usepackage{amssymb}
\usepackage{soul}
\usepackage{thm-restate}
\usepackage{graphicx} % Required for inserting images
\usepackage{cite}

\usepackage{libertine}

\usepackage{mathtools}
\usepackage{xcolor}
\usepackage{tikz}
\usetikzlibrary{calc}
\usetikzlibrary{shapes,decorations}
\usetikzlibrary{hobby}
\usetikzlibrary{decorations.markings,intersections}
\usetikzlibrary{arrows.meta}

\pgfdeclarelayer{background}
\pgfdeclarelayer{foreground}
\pgfsetlayers{background,main,foreground}

\usepackage{titling}
\usepackage{dsfont}
\usepackage{xspace}
\usepackage{stmaryrd}
\usepackage{mathrsfs}
\usepackage{hyperref}
\usepackage{cleveref}
\usepackage{comment}
\usepackage[inline]{enumitem}
\setlist{nosep}

\usepackage[subrefformat=parens]{subcaption} %needed for subfigures
\usepackage{nicefrac}
\usepackage{MnSymbol} %for vec
\usepackage{marginnote}
\usepackage{arxiv_macros}
\usepackage{tikzmacros}
\usepackage{textpos}

\usepackage{thm-restate}

\hypersetup{
	colorlinks=true,
	linkcolor=myBlue,
	citecolor=myBlue,
	urlcolor=myBlue,
	bookmarksopen=true,
	bookmarksnumbered,
	bookmarksopenlevel=2,
	bookmarksdepth=3
}
\newif\ifcomment
\commentfalse
\title{On graphs coverable by chubby shortest paths}
\author{
Meike Hatzel\thanks{Discrete Mathematics Group, Institute for Basic Science (IBS), Daejeon, South Korea. E-mail: \href{research@meikehatzel.com}{research@meikehatzel.com}. Meike Hatzel's research was supported by the Institute for Basic Science (IBS-R029-C1).}
\and
Micha{\l} Pilipczuk\thanks{Institute of Informatics, University of Warsaw, Poland. E-mail: \href{michal.pilipczuk@mimuw.edu.pl}{michal.pilipczuk@mimuw.edu.pl}. Micha\l{} Pilipczuk's research was supported by the project BOBR that has received funding from the European Research Council (ERC) under the European Union’s Horizon 2020 research and innovation programme, grant agreement No. 948057.}}
\date{}

\renewcommand{\phi}{\varphi}

\renewcommand{\hat}{\widehat}

\begin{document}

\maketitle

\begin{abstract}
    Dumas, Foucaud, Perez, and Todinca~[SIAM J.~Disc.~Math., 2024] proved that if the vertex set of a graph $G$ can be covered by $k$ shortest paths, then the pathwidth of $G$ is bounded by $\Oh{k \cdot 3^k}$.
    We prove a coarse variant of this theorem: if in a graph $G$ one can find~$k$ shortest paths such that every vertex is at distance at most $\rho$ from one of them, then $G$ is $(3,12\rho)$-quasi-isometric to a graph of pathwidth $k^{\mathcal{O}(k)}$ and maximum degree $\mathcal{O}(k)$, and $G$ admits a path-partition-decomposition whose bags are coverable by $k^{\mathcal{O}(k)}$ balls of radius at most $2\rho$ and vertices from non-adjacent bags are at distance larger than $2\rho$.
    We also discuss applications of such decompositions in the context of algorithms for finding maximum distance independent sets and minimum distance dominating sets in graphs.
\end{abstract}

\begin{textblock}{20}(-1.7, 6.6)
\includegraphics[width=40px]{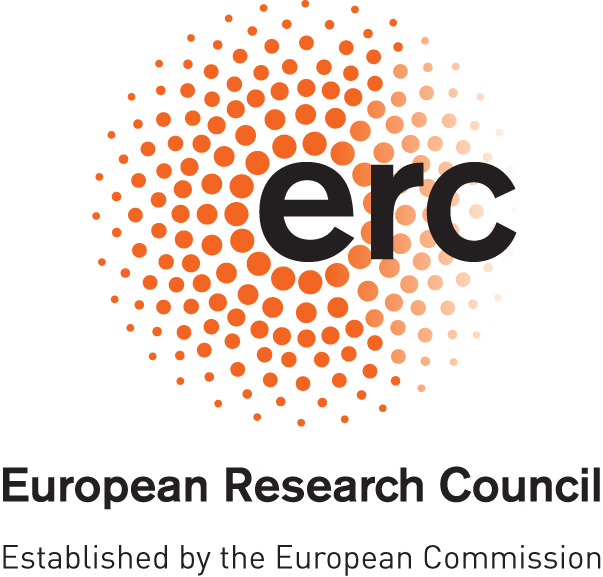}%
\end{textblock}
\begin{textblock}{20}(-1.7, 7.6)
\includegraphics[width=40px]{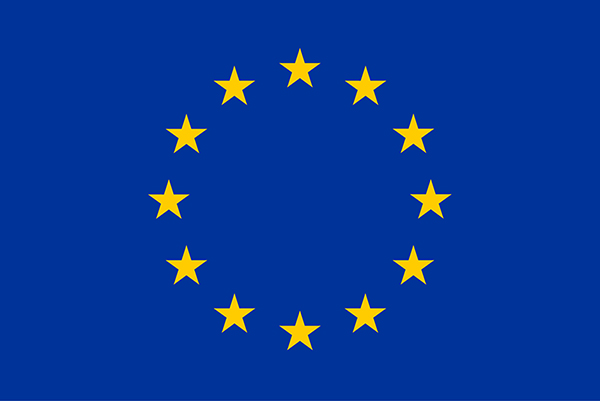}%
\end{textblock}

\section{Introduction}\label{sec:intro}

The aim of \emph{coarse graph theory} is to study the metric structure in graphs. Recently, Georgakopoulous and Papasoglu~\cite{coarse2023} launched a systematic investigation of the coarse counterpart of the theory of Graph Minors, with the following overarching conjecture in mind: If a graph $G$ excludes some fixed graph as a \emph{fat minor}, then the metric structure of $G$ resembles that of a graph excluding a classic minor (there are subtleties regarding this statement, see the discussion in~\cite{DaviesHIM24}). Here, fat minors are a coarse counterpart of classic minors, and the resemblance is measured through the existence of a \emph{quasi-isometry}: a roughly bijective mapping between graphs that roughly preserves distances.

The investigation of the theory of coarse minors naturally leads to studying coarse counterparts of classic graph parameters. The general principle of the coarse theory is that disjointness corresponds to farness, and intersection corresponds to closeness.
Hence, in the coarse counterpart of \treewidth, we would consider tree decompositions where the bags are not of bounded size but can be covered by a bounded number of balls of bounded radius. Very recently, Nguyen, Scott, and Seymour~\cite{coarsetw2025}, and independently Hickingbotham~\cite{hickingbotham2025twquasiisom} proved that the existence of such a tree decomposition is equivalent to being quasi-isometric to a graph of bounded \treewidth, and the natural analogue of this statement also works for pathwidth~\cite{hickingbotham2025twquasiisom}. Also very recently, Abrishami, Czyżewska, Kluk, Pilipczuk, Pilipczuk, and Rzążewski~\cite{abrishami2025coarsetreedecompositionscoarse} studied the connection between such tree decompositions and the existence of balanced separators consisting of a bounded number of bounded-radius~balls.

The drive permeating the recent work on coarse graph theory is to understand which results from classic structural graph theory can be lifted to the coarse setting and to what extent. See~\cite{coarse2023} for a broad overview and~\cite{AhnGHK25,AlbrechtsenHJKW24,BergerS24,DaviesHIM24,DujmovicJMM24,hickingbotham2025twquasiisom,coarsetw2025,NguyenSS25} for examples of individual works of this kind. In this paper, we add one more result to this growing list. Precisely, we examine the following elegant theorem proved recently by Dumas, Foucaud, Perez, and Todinca~\cite{dumas2024geodesics}.

\begin{theorem}[{\hspace{-0.3pt}\cite{dumas2024geodesics}}]
    \label{thm:cover}
    Suppose $G$ is a graph in which one can find a family $\cal P$ of $k$ shortest paths so that every vertex of $G$ belongs to some path of $\cal P$. Then $G$ has pathwidth bounded by $\Oh{k\cdot 3^k}$.
\end{theorem}

The main contribution of this work is the following coarse counterpart of \cref{thm:cover}. Here, for positive integers $k,\rho$, we call a graph $G$ \emph{$(k,\rho)$-geodesic-coverable} if in $G$ there is a family $\cal P$ of $k$ shortest paths so that every vertex is at distance at most $\rho$ from some path of $\cal P$.

\begin{restatable}{theorem}{thmMain}\label{thm:main}
    Suppose $G$ is a $(k,\rho)$-geodesic-coverable graph, for some positive integers $k,\rho$. Then $G$ is $(3,12\rho)$-quasi-isometric to a graph of pathwidth at most $k^{\Oh{k}}$ and maximum degree less than~$26k$. Furthermore, the vertices of $G$ can be partitioned into parts $C_1, C_2, \ldots, C_\ell$ so that every part can be covered by $k^{\Oh{k}}$ balls of radius $2\rho$, and every two vertices at distance at most $2\rho$ belong to the same part or to two consecutive parts.
\end{restatable}

The proof of~\cref{thm:main} follows the strategy proposed by Dumas et al.~\cite{dumas2024geodesics} and lifts it naturally to the coarse setting. Namely, we fix a root vertex $v$ and partition the graph into \textsl{breadth first search} (BFS) layers starting from $v$. The key technical statement is the following: every BFS layer $S$ can be covered by $k^{\Oh{k}}$ balls of radius $2\rho$. Inspired by~\cite{dumas2024geodesics}, the idea here is to classify the shortest paths from $v$ to the vertices of $S$ according to their \emph{types}.
Very roughly, every shortest path from $v$ to a vertex of $S$ can be technically massaged so that it first travels for a significant distance along one shortest path $P_1$ from the family $\Pp$ witnessing geodesic-coverability, then jumps through a short ``connector'' to another shortest path $P_2\in \Pp$ and travels for a significant distance along it, and so on; the sequence $P_1, P_2,\ldots$ is the type of $Q$.
There are only $k^{\Oh{k}}$ relevant types to consider, and it turns out that paths of the same type starting in the same vertex lead to vertices that are at distance at most $k^{\Oh{k}}\cdot \rho$ from each other. This allows us to cover $S$ with $k^{\Oh{k}}$ balls of radius $k^{\Oh{k}}$, and an additional argument is needed to cover each of those balls with $k^{\Oh{k}}$ balls of radius $2\rho$. Once the coverage property of the BFS layers is established, we naturally construct the partition $C_1, C_2, \ldots, C_\ell$ so that each $C_i$ consists of $2\rho$ consecutive layers. The quasi-isometry statement is obtained by combining all these observations with a generic construction. We note that the proof naturally gives a linear-time algorithm to construct a partition $C_1, C_2, \ldots, C_\ell$ with the asserted properties.

Compared to \cref{thm:cover}, \cref{thm:main} gives a $k^{\Oh{k}}$ bound instead of $2^{\Oh{k}}$. In \cite{dumas2024geodesics}, Dumas et al.~first gave a simpler argument giving a bound of $k^{\Oh{k}}$ on the number of relevant types and then proceeded with a more complicated argument reducing the bound to $2^{\Oh{k}}$.
In our proof, we use the simpler argument because we lose an additional $k^{\Oh{k}}$ factor later in the proof as well when proving that paths of the same type lead to vertices not far from each other.
Compared to \cref{thm:cover}, this part of the proof gets far more complicated in the coarse setting.

%\medskip

The partition $C_1,C_2,\ldots,C_\ell$ as described in \cref{thm:main} shall be called a \emph{distance-$2\rho$ path-partition-decomposition} of $G$. We see it as the natural coarse counterpart of the linear variant of \emph{tree-partition-decompositions}, underlying the graph parameter \emph{tree-partition width}; see an overview of this parameter e.g.~in the work of Wood~\cite{Wood09}. More generally, a \emph{distance-$\rho$ tree-partition-decomposition} of a graph $G$ would be a tree $T$ together with a partition of the vertex set of $G$ into parts $\Set{C_x \colon x\in \V{T})}$ so that whenever $\dist{}{u,v}\leq \rho$, $u$ and $v$ must belong to the same part $C_x$ or to two parts $C_x, C_y$ so that $x,y$ are adjacent in $T$. We observe that such a decomposition can be used to efficiently solve some computational problems of metric character in $G$, by means of dynamic programming. We showcase this on the examples of the computation of $\distIS{2\rho}{G}$ --- the maximum number of disjoint balls of radius~$\rho$ that can be packed in $G$ --- and $\distDS{\rho}{G}$ --- the minimum number of balls of radius~$\rho$ needed to cover all the vertices of~$G$.

\begin{restatable}{theorem}{thmDistIS}
    \label{thm:distIS_xp}
    There is an algorithm that, given an $n$-vertex graph $G$ and a \mbox{distance-$\rho$} tree-partition-decomposition of $G$ of width at most $k$ and degree at most~$\Delta$, computes $\distIS{2\rho}{G}$ in time~$n^{\bigO{\Delta k}}$.
\end{restatable}

\vspace{-0.4cm}
%restatable puts incorrectly large spaces between consecutive theorems, putting a negative vspace to counter this 

\begin{restatable}{theorem}{thmDistDS}
    \label{thm:distDS_xp}
    There is an algorithm that, given an $n$-vertex graph $G$ and a \mbox{distance-$\rho$} tree-partition-decomposition of $G$ of width at most $k$ and degree at most~$\Delta$, computes $\distDS{\rho}{G}$ in time~$n^{\bigO{\Delta^2 k}}$.
\end{restatable}

By combining \cref{thm:main,thm:distIS_xp,thm:distDS_xp}, we obtain the following algorithmic consequence.

\begin{corollary}
Suppose $G$ is an $n$-vertex $(k,\rho)$-geodesic-coverable graph, for some positive integers~$k,\rho$. Then $\distIS{4\rho}{G}$ and $\distDS{2\rho}{G}$ can be computed in time $n^{k^{\Oh{k}}}$.
\end{corollary}

In the language of \textsl{parameterized complexity}, this places the computation of $\distIS{4\rho}{G}$ and $\distDS{2\rho}{G}$ in the complexity class $\mathsf{XP}$ (\emph{slicewise polynomial}) when parametrised by the smallest $k$ such that the graph is $(k,\rho)$-geodesic-coverable.

%\paragraph*{Outline.} In this extended abstract we focus on proving \cref{thm:main}. \cref{thm:distIS_xp} and \cref{thm:distDS_xp} are proved in \cref{sec:algorithms}.
%Proofs of some statements (marked with \app) have been moved to \cref{app:proofs} due to space constraints.

\section{Preliminaries}
\label{sec:preliminaries}

\paragraph{Paths and walks.}
A \emph{walk} $P$ of length $\ell$ in a graph $G$ is a sequence of vertices $v_0, \dots, v_{\ell}$ such that $v_iv_{i+1} \in \E{G}$ for all $0 \leq i < \ell$; and $P$ is a \emph{path} if the vertices $v_0,\ldots,v_\ell$ are additionally pairwise different.
We sometimes identify $P$ with the subgraph of $G$ containing these edges and vertices.
We identify $v_0$ as the \emph{startvertex}, $\Start{P}$, and $v_{\ell}$ as the \emph{endvertex}, $\End{P}$, in order to distinguish the two ways of traversing a walk: \emph{forward} from $\Start{P}$ to $\End{P}$, or \emph{backward} from $\End{P}$ to $\Start{P}.$
We refer to $P$ as a $v_0$-$v_{\ell}$-walk (or $v_0$-$v_\ell$-path, if it is a path).
If $x,y \in \V{P}$, then we write $xPy$ for the subwalk of $P$ starting in $x$ and ending in $y$.
We write $xP$ instead of $xP\End{P}$ and $Px$ instead of $\Start{P}Px$.
For two walks $P$ and $Q$ with $\End{P}=\Start{Q}$, we write $P \cdot Q$ for their concatenation.
We also combine these two notations to write $PxQ$ instead of $Px \cdot xQ$.
%\Me{I suggest just adding a note here that this still works for walks and vertices occuring at most once on them and then putting the conditions of the $R_i$ being disjoint from all visited geodesics back in.}

\paragraph{Metric aspects.}
The \emph{distance} between two vertices $u, v \in \V{G}$, denoted $\dist{G}{x,y}$, is the length of the shortest $u$-$v$-path in $G$.
We drop the index in case the graph is evident from the context.
We extend the notation to sets $X,Y$ of vertices in the expected way: $\dist{G}{u,X}\coloneqq \min_{v\in X} \dist{G}{u,v}$ and $\dist{G}{Y, X}\coloneqq \min_{u\in Y} \dist{G}{u, X}$.
If $P$ and $Q$ are paths, we simply write $\dist{}{u,P}$ and $\dist{}{Q,P}$ instead of $\dist{}{u,\V{P}}$ and $\dist{}{\V{Q},\V{P}}$.
A \emph{geodesic} in $G$ is a path that is a shortest path between its endpoints.
Note that every subpath of a geodesic is a geodesic as well.
A path/walk $P$ is a \emph{$\gamma$-almost shortest path/walk} if $\Abs{\Abs{P} - \dist{}{\Start{P},\End{P}}} \leq \gamma$.
Note that every subpath/subwalk of a $\gamma$-almost shortest path/walk is again a $\gamma$-almost shortest path/walk.

For a vertex $u$ and a non-negative integer $\rho$, the radius-$\rho$ ball around $v$ is
\begin{equation*}
    \ball{\rho}{v} \coloneqq \Set{u\in \V{G} \mid \dist{}{u,v} \leq \rho}.
\end{equation*}

A collection of paths $\mathcal{P}$ is a \emph{distance-$\rho$ path-cover} if every vertex of $G$ lies at distance at most $\rho$ of some vertex of a path in $\mathcal{P}.$
%A path $P$ is a \emph{geodesic} if it is the shortest path between $\Start{P}$ and $\End{P}$.
We say that $G$ is \emph{$(k,\rho)$-geodesic coverable} if there is a distance-$\rho$ path-cover $\mathcal{P}$ of size $k$ and all paths in $\mathcal{P}$ are geodesics. We then also say that $\mathcal{P}$ is a \emph{distance-$\rho$ geodesic-cover}, or a \emph{$(k,\rho)$-geodesic-cover} of $G$.

A \emph{distance-$\rho$ independent set} in $G$ is a set $I \subseteq \V{G}$ such that $\dist{}{u,v} > \rho$ for all distinct $u,v \in I$.
The maximum size of a distance-$\rho$ independent subset of a set $A\subseteq V(G)$ is denoted by $\distIS{\rho}{A,G}$, and the distance-$\rho$ independence number of $G$ is defined as $\distIS{\rho}{G}\coloneqq \distIS{\rho}{\V{G},G}$. %A {\em{maximum distance-$\rho$ independent set}} is one of maximum possible size, i.e. $\distIS{\rho}{G}$.

A \emph{distance-$\rho$ dominating set} of a set $A\subseteq \V{G}$ in a graph $G$ is a set $D \subseteq \V{G}$ such that for every $u \in A$ there exists $d \in D$ satisfying $\dist{}{u,d} \leq \rho$; we also say that $D$ {\em{distance-$\rho$ dominates}} $A$.
The minimum size of a distance-$\rho$ dominating set of $A$ is denoted as $\distDS{\rho}{A,G}$. We say that $A$ is \emph{$\Brace{k,\rho}$-coverable} if $\distDS{\rho}{A,G}\leq k$.
The \emph{distance-$\rho$ domination number} of $G$ is defined as $\distDS{\rho}{G}\coloneqq \distDS{\rho}{\V{G},G}$. %A {\em{minimum distance-$\rho$ dominating set}} is one of minimum possible size, i.e. $\distDS{\rho}{G}$.

We note the following simple relation between the quantities defined above.

\begin{restatable}{lemma}{ISDS}
    \label{lem:is-ds-comparison}
    For any graph $G$, set $A\subseteq \V{G}$, and positive integer~$\rho$, we have    
        $\distIS{2\rho}{A,G}\leq \distDS{\rho}{A,G}$.
\end{restatable}
\begin{proof}
	Let $I\subseteq A$ be a maximum-size distance-$2\rho$ independent subset of $A$, and $D$ be a minimum-size distance-$\rho$ dominating set of $A$. Observe that for every $v\in D$, there is at most one vertex $u\in I$ such that $\dist{}{u,v}\leq \rho$, for otherwise, by triangle inequality, we would have two vertices of $I$ at distance at most $2\rho$ from each other. Since every vertex of $I$ is at distance at most $\rho$ from some vertex of~$D$, this proves that $|I|\leq |D|$.
\end{proof}

\paragraph{Decompositions.} We make use the following coarse counterparts of \emph{path-partition-decompositions} and \emph{tree-partition-decompositions}.
Similarly to other notions in coarse graph theory, these are obtained by requiring the parts of the decomposition to be coverable by a small number of balls.
Fix a positive integer $\rho$.

\begin{definition}
A \emph{distance-$\rho$ path-partition-decomposition} of a graph $G$ is a sequence $C_1,\dots, C_\ell$ of subsets of $\V{G}$ such that
\begin{itemize}
    \item $ \{C_1,C_2, \ldots,C_\ell\}$ is a partition of $\V{G}$, and
    \item for all $u,v \in \V{G}$ with $\dist{}{u,v} \leq \rho$ and $u \in C_i, v \in C_j$, we have $\Abs{i-j} \leq 1$.
\end{itemize}
The \emph{width} of this path-partition-decomposition is defined as
$\max_{1 \leq i \leq n} \Set{\distDS{\rho}{C_i,G}}$.
\end{definition}

\begin{definition}
A \emph{distance-$\rho$ tree-partition-decomposition} of a graph $G$ consists of a tree $T$ and, for every node $x\in \V{T}$, a set of vertices $C_x\subseteq \V{G}$ such that
\begin{itemize}
    \item $\Set{C_x\colon x\in \V{T}}$ is a partition of $\V{G}$, and
    \item for all $u,v \in \V{G}$ with $\dist{}{u,v} \leq \rho$ and $u \in C_x, v \in C_y$, the nodes $x$ and $y$ are either equal or adjacent in $T$.
\end{itemize}
The {\em{width}} of this tree-partition-decomposition is
$\max_{1 \leq i \leq n} \Set{\distDS{\rho}{C_i,G}}$, and its \emph{degree} is the maximum degree of $T$.
\end{definition}

Thus, the width of a path- or tree-partition-decomposition is at most $k$ if and only if every part $C_x$ is $(k,\rho)$-coverable. Let us also recall the classic notions of \emph{path-decompositions} and \emph{pathwidth}.

\begin{definition}
    \label{def:pathwidth}
    A \emph{path-decomposition} of a graph $G$ is a sequence $D_1, \dots, D_{\ell} $ of subsets of $\V{G}$ satisfying the following:
    \begin{itemize}
        \item for every $e \in \E{G}$ there is an $i \in [\ell]$ with $e \subseteq D_i$, and
        \item for all $i \leq j \leq k$ holds $D_i \cap D_k \subseteq D_j.$
    \end{itemize}
    The \emph{width} of such a decomposition is the size of its largest set, and the \emph{pathwidth} of $G$ is the minimal width over all path-decompositions of $G$.
\end{definition}

\paragraph{Quasi-isometries.}
Let $G$ and $H$ be graphs.
A function $\phi~\colon \V{G} \to \V{H}$ is an \emph{$(m,a)$-quasi-isometry from $G$ to $H$} if 
for all $u,v \in \V{G}$ we have
\begin{equation*}
	m^{-1} \cdot \dist{G}{u,v} - a \leq \dist{H}{\Fkt{\phi}{u},\Fkt{\phi}{v}} \leq m \cdot \dist{G}{u,v} + a,
\end{equation*}
and for every $w\in \V{H}$ there exists $u\in \V{G}$ with
    $\dist{H}{w,\Fkt{\phi}{u}} \leq a$.
We say $G$ is \emph{$(m,a)$-quasi-isometric} to $H$ if there is an $(m,a)$-quasi-isometry from $G$ to~$H$.

We use the following generic construction of a quasi-isometry between a graph, and its \say{coarsening} induced by a maximal distance independent set. We use the terminology from~\cite{abrishami2025coarsetreedecompositionscoarse}, but the method is folklore; see e.g.~\cite[Observation~2.1]{coarse2023} or~\cite[Theorem~4]{AlbrechtsenHJKW24}.

Let $G$ be a graph and $I$ be a maximal distance-$\rho$ independent set in $G$.
The \emph{$(I,\rho)$-distance graph} of~$G$, denoted $\DistGraph{\rho}{I}{G}$, is defined as follows
\begin{itemize}
	\item $\V{\DistGraph{\rho}{I}{G}} \coloneqq I$, and
	\item for all $x,y \in I$ we have $xy \in \E{\DistGraph{\rho}{I}{G}}$ if and only if $\dist{G}{x,y} \leq 3\rho$.
\end{itemize}
Note that $\DistGraph{\rho}{I}{G}$ is unweighted, so single edges in this graph roughly correspond to distances between $\rho$ and $3\rho$ in $G$. To get a closer connection between the shortest path metrics, we also define $\whDistGraph{\rho}{I}{G}$ as the graph obtained from $\DistGraph{\rho}{I}{G}$ by replacing every edge with a path of length $3\rho$. With this definition, we have that $G$ and $\whDistGraph{\rho}{I}{G}$ are quasi-isometric in the sense below.

\begin{restatable}[{\hspace{-0.3pt}\cite{abrishami2025coarsetreedecompositionscoarse}}]{lemma}{distGraph}\label{thm:distance-graph}
	Let $G$ be a graph and $I$ be an inclusionwise maximal distance-$\rho$ independent set in $G$. For every $u\in \V{G}$, let 
    $\Fkt{\varphi}{u}$ be a vertex of $I$ closest to $u$ in $G$. (If there are several such vertices, choose an arbitrary one.) Then $\phi$ is a $\Brace{3,3\rho}$-quasi-isometry from $G$ to $\whDistGraph{\rho}{I}{G}$.
\end{restatable}

We note that in~\cite{abrishami2025coarsetreedecompositionscoarse}, Abrishami et al. derive $\whDistGraph{\rho}{I}{G}$ from $\DistGraph{\rho}{I}{G}$ slightly differently: instead of replacing edges by paths of length $3\rho$, they just consider $\whDistGraph{\rho}{I}{G}$ to be an edge-weighted graph obtained from $\DistGraph{\rho}{I}{G}$ by assigning weight $3\rho$ to every edge. This minor difference is immaterial for the proof of \cref{thm:distance-graph}.

%Since the work of Abrishami et al.~\cite{AbrishamiCzKPPRz25} is not yet publicly available, we include the (short) proof of \cref{thm:distance-graph} from \cite{AbrishamiCzKPPRz25} in \cref{app:proofs}.

\section{\Snappaths}

\newcommand{\concat}[1]{\Fkt{\variablestyle{concat}}{#1}}

In this section we discuss the concept of \emph{\snappaths}. Intuitively, a \snappath in a graph $G$ is a sequence consisting alternately of subpaths of geodesics from a fixed geodesic-cover $\Pp$ of $G$, and short ``connectors'' joining the endpoints of those subpaths in order. The idea is that every shortest path $Q$ in $G$ can be ``snapped'' to a \snappath by (roughly) snapping every vertex on $Q$ to the closest geodesic from~$\Pp$. However, we want to avoid reusing the same geodesic of $\Pp$ several times in a \snappath, so every segment between two consecutive ``snappings'' to some $P\in \Pp$ is replaced by simply sliding along~$P$. 

We proceed to a formal description.
Let $G$ be a $(k,\rho)$-geodesic-coverable graph and $\mathcal{P}$ a $(k,\rho)$-geodesic-cover of $G$.
A \emph{$\rho$-\snappath} $Q$ is a sequence $(R_0 , Q_1 , R_1 , Q_2,$ $R_2 , \dots , Q_{k'} , R_{k'})$ with $k' \leq k$ such~that
\begin{itemize}[nosep]
    \item every $Q_i$ is a subpath of a geodesic in $\mathcal{P}$;
    \item for every $P \in \Pp$ there is at most one $i$ such that $Q_i$ is a subpath of $P$;
    \item every $R_i$ is walk in $G$ of length at most $2\rho+1$, and moreover, $R_0$ and $R_{k'}$ are of length at most~$\rho$; and
    \item $\End{R_{i-1}}=\Start{Q_i}$ and $\End{Q_i}=\Start{R_i}$, for all $i\in \Set{1,\ldots,k'}$.
%    \item every $R_i$ is internally disjoint to all geodesics in $\mathcal{P}$,
%    \item every $R_i$ with $i\in \Set{1,\ldots,k'-1}$ has length at most $2 \rho + 1$, and
%    \item every $R_i$ with $i \in \Set{0,k'}$ has length at most $\rho$.
\end{itemize}
We define $\concat{Q} \coloneqq R_0 \cdot Q_1 \cdot R_1 \cdot Q_2 \cdot R_2 \cdot \dots \cdot Q_{k'} \cdot R_{k'}$.
We generalise notation for paths and walks to \snappaths by applying it to the concatenation instead wherever sensible, for example, $\Abs{Q} \coloneqq \Abs{\concat{Q}}$.
The geodesics of $\Pp$ that the paths $Q_i$ are subpaths of are called {\em{visited}} by $Q$.

The following lemma establishes that every path in $G$ lies \say{near to} a \snappath.

\begin{lemma}
	\label{lem:every_path_can_be_snapped}
    Let $G$ be a $(k,\rho)$-geodesic-coverable graph and $\mathcal{P}$ a $(k,\rho)$-geodesic-cover of $G$.
	For every shortest path $P$ in $G$ there exists a $\rho$-\snappath $Q$ with $\Start{Q} = \Start{P}$, $\End{Q} = \End{P},$ and $\Abs{\Abs{P} - \Abs{Q}} \leq 4\rho k$.
\end{lemma}
\begin{proof}
    Let $\nearestGeodesicVtx{}\colon \V{G} \to \bigcup_{P_i \in \mathcal{P}} \V{P_i}$ map every vertex of $G$ to a vertex on a path in $\mathcal{P}$ minimising the distance to it. (If there are multiple vertices minimising the distance, select an arbitrary one.) Note that $\dist{}{x,\nearestGeodesicVtx{x}}\leq \rho$, for $\Pp$ is a distance-$\rho$ geodesic-cover.
    Also, let $\nearestGeodesic{}\colon \V{G} \to \mathcal{P}$ map every vertex $x$ to a path of $\mathcal{P}$ containing $\nearestGeodesicVtx{x}$, \ie~$\nearestGeodesicVtx{x} \in \nearestGeodesic{x}$. For a walk $S$, we write $\nearestGeodesic{S}\coloneqq \nearestGeodesic{\V{S}}$.

    We now prove the following statement, which implies the desired one, by induction on $k'\coloneqq \Abs{\nearestGeodesic{P}}$:
    There is a $\rho$-\snappath $Q$ with $\Start{Q} = \Start{P}$, $\End{Q} = \End{P}$, and $\Abs{\Abs{P} - \Abs{Q}} \leq 4\rho k'$ which only visits paths in $\nearestGeodesic{P}$.

    If $\nearestGeodesic{x} = \nearestGeodesic{y} = P_1 \in \mathcal{P}$ for all $x,y\in\V{P}$, then there is a path $R_0$ of length at most~$\rho$ from $\Start{P}$ to $\nearestGeodesicVtx{\Start{P}}$ and a path $R_1$ from $\End{P}$ to $\nearestGeodesicVtx{\End{P}}$.
    Then we observe that the walk $R_0 \cdot \nearestGeodesicVtx{\Start{P}}P_1 \nearestGeodesicVtx{\End{P}}\cdot R_1$ yields the desired \snappath.

%    \Mi{In the paragraph below: What if $v$ is the end of $P$?}
%    \Me{Then $P_1$ is the only geodesic in $\nearestGeodesic{\V{P}}$ and this is handled in the IB.}

%and internally disjoint from all geodesics in $\mathcal{P}$.

    So assume $\nearestGeodesic{P}$ contains at least two geodesics.
    Let $u \coloneqq \Start{P}$ and $P_1 \coloneqq \nearestGeodesic{u}$.
    Let $v \in \V{P}$ be the last vertex along $P$ with $\nearestGeodesic{v} = P_1$ and $v'$ be its successor on $P$ (if $v=\End{P}$, we proceed as in the base case).
    We define $R_0$ to be a shortest path from $u$ to $\nearestGeodesicVtx{u}$ and $R^\circ_1$ to be a shortest path from $v$ to~$\nearestGeodesicVtx{v}$; both are of length at most $\rho$.
    Also, let $Q_1 \coloneqq \nearestGeodesicVtx{u} P_1 \nearestGeodesicVtx{v}$.
    Note that, as both $Pv$ and $Q_1$ are shortest paths, we have $\Abs{\Abs{Pv} - \Abs{Q_1}} \leq 2\rho.$
    As $\nearestGeodesic{v'P}$ contains a geodesic less than $\nearestGeodesic{P}$, by induction, there is a $\rho$-\snappath $Q'=(R'_0,Q_1',R_1',\ldots,Q'_{k'-1},R_{k'-1})$ starting in $v'$ and ending in $\End{P}$ with $\Abs{\Abs{v'P} - \Abs{Q'}} \leq 4\rho(k'-1)$ and only using geodesics from $\nearestGeodesic{v'P}$, so in particular not using $P_1$.
    Next, we define \[Q \coloneqq (R_0,\, Q_1,\, R^\circ_1 \cdot vv' \cdot R_0',\, Q'_1,\ldots,\, Q'_{k'-1},\, R_{k'-1}).\]
    Note that $|R_0|\leq \rho$ and $|R_1^\circ\cdot vv'\cdot R_0'|\leq 2\rho+1$. Further, we have\[\Abs{\Abs{P} - \Abs{Q}} \leq \Abs{R_0}+\Abs{\Abs{Pv} - \Abs{Q_1}} + \Abs{R^\circ_1} + \Abs{\Abs{v'P} - \Abs{Q'}} \leq 4\rho + 4\rho(k'-1) = 4\rho k'.\]
    We conclude that $Q$ satisfies all the requirements for the desired $\rho$-\snappath.
\end{proof}

%\begin{figure}
%    \includegraphics[scale=0.3]{Figures/snap_pic.png}
%    \caption{Dupa}
%\end{figure}
%    Now, let $R'_0$ be the maximal prefix of $Q'$ that is internally disjoint from all geodesics in $\mathcal{P}$ and let $R_1$ be the shortest $\nearestGeodesicVtx{v}$-$\End{R'_0}$-path in $\InducedSubgraph{G}{R'_0 \cup R'_1}$.
%    Note that $\Abs{R_1} \leq 2\rho + 1.$
%    We replace $\Start{R_1} \hat{Q}\, \End{R_1}$ in $\hat{Q}$ with $R_1$ obtaining the path $Q$.
%    Note that $\Abs{Q} \leq \Abs{\hat{Q}}$.
%    Thus, $Q$ yields the desired $\rho$-\snappath.

% % Based on \cref{lem:snappaths_visit_geodesics_at_most_once} we can define a type for the \snappaths as follows.
% Let $\mathcal{P}$ be a $(k,\rho)$-geodesic-cover of a graph $G$ and $Q = R_0 \cdot Q_1 \cdot R_1 \cdot Q_2 \cdot R_2 \cdot \dots \cdot Q_{\ell} \cdot R_{\ell}$ be a $\rho$-\snappath.
% For every $Q_i$, let $P_i \in \mathcal{P}$ be the geodesic $Q_i$ is a subpath of.
% If $Q_i$ follows $P_i$ forwards, we write $\forward{P_i}$, otherwise we write $\backward{P_i}$.
% This yields a tuple $(\overline{P}_1, \dots, \overline{P}_\ell)$ with $\overline{P}_i \in \Set{\forward{P_i},\backward{P_i}}$, which we refer to as the \emph{type} of $Q$, written $\type{Q}$.

%\Me{Don't we need sth like $\rho_i \coloneqq \rho_{i-1}(4k + k^2)+k^2\rho.$ I think $(2k + 2)$ only suffices to account for the old connectors, but not for the old segments that now become part of the connectors.
%I am worried, should these not actually be up to order of $\rho_i$?}

\newcommand{\newstuff}[1]{{\leavevmode\color{blue}[#1]}}
\newcommand{\deltaSW}{\delta_{\variablestyle{SW}}}

Next, we need to massage the \snappaths a bit in order to achieve a certain minimality property. Fix a graph $G$ with a $(k,\rho)$-geodesic-cover $\Pp$. We define a sequence of values $\rho_0,\ldots,\rho_k$ as follows:
\begin{align*}
    \rho_0 \coloneqq 2\rho+1  \qquad\textrm{and}\qquad  \rho_i \coloneqq (2k+5)\cdot \rho_{i-1}\quad\textrm{for }i=1,2,\ldots,k.
\end{align*}
For $\ell\in \{0,1,\ldots,k\}$, an \emph{$\ell$-simplified $\rho$-\snappath} $Q$ is a sequence $(R_0, Q_1,R_1,Q_2, R_2,\dots,\allowbreak Q_{k'},R_{k'})$ such that
\begin{itemize}[nosep]
\item every $Q_i$ is subpath of a geodesic in $\Pp$, with every geodesic in $\Pp$ containing at most one $Q_i$ as a subpath,
\item every $R_i$ is a walk in $G$ of length at most $\rho_{\ell}$, and
\item every $Q_i$ is of length larger than $(2k+3)\rho_{\ell}$.
\end{itemize}
Note that we allow $k'=0$; then $Q$ consists only of the walk $R_0$. Again, the $k'$ paths of $\Pp$ that contain the paths $Q_i$ are subpaths are called {\em{visited}} by $Q$.
Same as for \snappaths, we generalise the notation for paths and walks to simplified \snappaths by applying it to the concatenation wherever sensible.

%, for example, $\Abs{Q} \coloneqq \Abs{\concat{Q}}$.

%\item every $R_i$ is internally disjoint to all elements of $\mathcal{P}$ such that there is a $Q_j$ that is a subpath of it,
%\item every $Q_i$ is of length more than .
%Note that while the $R_i$ can intersect geodesics of $\mathcal{P}$, we only consider those that contain the $Q_i$ as subpaths as \emph{visited} by $Q$.

The intuition behind simplified \snappaths is that we allow the walks $R_i$ to be somewhat longer, but in return we require that all paths $Q_i$ are significantly longer than all the walks $R_i$. In the next lemma we show that every \snappath can be turned into a simplified \snappath by merging some walks in the~sequence.

\begin{lemma}
    \label{lem:every_snal_path_can_be_simplified}
    Let $G$ be a graph with a $(k,\rho)$-geodesic-cover $\Pp$.
    Then for every $\rho$-\snappath $P$, there is some $\ell\in \Set{0,1,\ldots,k}$ and an $\ell$-simplified $\rho$-\snappath $P'$ such that $\concat{P}= \concat{P'}$.
\end{lemma}
\begin{proof}
    Let an \emph{$\ell$-presimplified $\rho$-\snappath} be defined exactly as an $\ell$-simplified $\rho$-\snappath, except we do not require the last item of the definition --- the lower bound on the lengths of the paths $Q_i$. Thus, every $\rho$-\snappath, in particular also~$P$, is a $0$-presimplified $\rho$-\snappath. Denote $P_0\coloneqq P$ and perform the following iterative construction of $P_1,P_2,\ldots$, where each $P_i$ is an $i$-presimplified $\rho$-\snappath and $\concat{P}=\concat{P_1}=\concat{P_2}=\ldots$:
    \begin{itemize}%[nosep]
        \item If $P_i$ is an $i$-simplified $\rho$-\snappath, finish the construction by setting $P'\coloneqq P_i$.
        \item Otherwise, denoting $P_i=(R_0,Q_1,R_1,\ldots,R_{k'-1},Q_{k'},R_{k'})$, we have
        $|Q_j|\leq (2k+3)\rho_i$ for some $j\in \Set{1,\ldots,k'}$.
        Obtain $P_{i+1}$ from $P_i$ by replacing the walks $R_{j-1},Q_j,R_j$ with their concatenation $R_{j-1}\cdot Q_j\cdot R_j$. Observe that
        \begin{equation*}
            |R_{j-1}\cdot Q_j\cdot R_j|\leq \rho_i+(2k+3)\rho_i+\rho_i = \rho_{i+1},
        \end{equation*}
        so $P_i$ is indeed an $(i+1)$-presimplified $\rho$-\snappath.
    \end{itemize}
    Each step of the construction reduces the length of the considered sequence of walks by $2$. As this length is at most $2k+1$ at the beginning (for~$P_0$), the construction makes at most $k$ steps in total and finds an $\ell$-simplified \snappath $P'$ with $\concat{P}=\concat{P'}$, for some $\ell\in \{0,1,\ldots,k\}$.
\end{proof}

Let $\mathcal{P}$ be a $(k,\rho)$-geodesic-cover of a graph $G$ and $Q = (R_0 , Q_1 , R_1 , \dots , Q_{k'} , R_{k'})$ be an $\ell$-simplified $\rho$-\snappath visiting $k'$ geodesics.
For every $Q_i$, let $P_i \in \mathcal{P}$ be the geodesic $Q_i$ is a subpath of.
If $Q_i$ follows $P_i$ forwards, we write $\forward{P_i}$, otherwise we write $\backward{P_i}$.
This yields a tuple $(\overline{P}_1, \dots, \overline{P}_{k'},\ell)$ with $\overline{P}_i \in \Set{\forward{P_i},\backward{P_i}}$. We call this tuple the \emph{type} of $Q$, written $\type{Q}$.

\section{Geodesic-coverability and coarse pathwidth}

In this section, we prove our main result, \cref{thm:main}. We first observe that in geodesic-coverable graphs, every ball can be covered with a relatively small number of balls of radius $2\rho$.

\begin{lemma}\label{lem:balls-coverable}
    Let $G$ be a $\Brace{k,\rho}$-geodesic-coverable graph.
    Then for all $v \in \V{G}$, the ball
    $\ball{\ell\rho}{v}$
    is $\Brace{2k(\ell+1),2\rho}$-coverable.
\end{lemma}
\begin{proof}
    Let $\mathcal{P}$ be a $(k,\rho)$-geodesic-cover of $G$ and $v \in \V{G}$.
    For every $P\in \mathcal{P}$, let $I_P$ be an inclusionwise maximal distance-$\rho$ independent subset of $\V{P}\cap \ball{(\ell+1)\rho}{v}$. Let $a$ and $b$ be the first and the last vertex of $I_P$, respectively, in the order of appearance on $P$. By triangle inequality, we have
    $\dist{}{a,b}\leq \dist{}{a,v}+\dist{}{v,b}\leq 2(\ell+1)\rho.$
    Since $P$ is a geodesic, the subpath of $P$ between $a$ and $b$ has length at most $2(\ell+1)\rho$. As all the vertices of $I_P$ lie on this subpath and $I_P$ is distance-$\rho$ independent, we conclude that $|I_P|\leq 2(\ell+1)$.

    Let 
    	$D\coloneqq \bigcup_{P\in \mathcal{P}} I_P$; then
     $|D|\leq |\mathcal{P}|\cdot 2(\ell+1)=2k(\ell+1)$. Note that for each $P\in \cal P$, the maximality of $I_P$ implies that $I_P$ distance-$\rho$ dominates $\V{P}\cap \ball{(\ell+1)\rho}{v}$; and the union of the sets $\V{P}\cap \ball{(\ell+1)\rho}{v}$ over all $P\in \cal P$ distance-$\rho$ dominates $ \ball{\ell\rho}{v}$, as $\cal P$ is a $(k,\rho)$-geodesic-cover. Hence $D$ distance-$2\rho$ dominates $\ball{\ell\rho}{v}$, and thus $D$ witnesses that $\ball{\ell\rho}{v}$ is $(2k(\ell+1),2\rho)$-coverable.
\end{proof}

The next technical  lemma contains the core argument of the proof: two simplified \snappaths that start at the same vertex, have the same type, and have roughly the same length must lead to vertices that are not too far from each other.

% \newstuff{
% \Me{Are these things true?}
% \begin{itemize}
%     \item $P$ $\delta$-almost short $\Rightarrow$ $P'$ subpath of $P$ $\delta$-almost short
%     \item $P \cdot Q$ $\delta$-almost short and $P'$ with $\Start{P}=\Start{P'}$, $\End{P} = \End{P'}$, and $\Abs{P'} = \Abs{P}+x$ $\Rightarrow$ $P' \cdot Q$ is $(\delta + x)$-almost short
%     \item a subpath of a simplified snappath is again a simplified snappath
% \end{itemize}
% }

% \Me{The bound in \cref{lem:same_start_and_same_type_impl_near_end} changed to $\gamma + k(6\rho +3)$}

\begin{lemma}
    \label{lem:same_start_and_same_type_impl_near_end}
	% Let $G$ be a graph with a $(k,\rho)$-geodesic-cover $\mathcal{P}$.
	% Suppose $P$ and $Q$ are two $\rho$-\snappaths such that
	% \begin{itemize}
	% 	\item there is a $P_i \in \mathcal{P}$ with $\Start{P} = \Start{Q} \in \V{P_i}$,
	% 	\item $\type{P} = \type{Q}$, and
	% 	\item $\Abs{\Abs{P} - \Abs{Q}} \leq \gamma$, for some non-negative integer $\gamma$.
	% \end{itemize}
	% Then $\dist{}{\End{P},\End{Q}} \leq \gamma + k(6\rho +3)$.

    Let $G$ be a graph with a $(k,\rho)$-geodesic-cover $\mathcal{P}$.
	Suppose $C$ and $D$ are two $\ell$-simplified $\rho$-\snappaths such that
	\begin{itemize}
		\item $\Start{C}=\Start{D}$ and the first walk ($R_0$) on both $C$ and $D$ is of length $0$,
		\item $\type{C} = \type{D}$,
		\item $\Abs{\Abs{C} - \Abs{D}} \leq \gamma$, for some non-negative integer $\gamma$, and
        \item $C,D$ are $\delta$-almost shortest paths with $\delta = 4k\rho +2\rho_\ell+ 2(k-k')\rho_{\ell}$ where $k'$ is the number of geodesics in $\type{C}$.
	\end{itemize}
	Then $\dist{}{\End{C},\End{D}} \leq \gamma + k\cdot (4k\rho+2\rho_\ell + 2k\rho_{\ell}) +k\cdot (4\rho_{\ell})$.
%    \Me{should the function in the lemma statement also rather depend on $k'$ than on $k$? that is $k'(2k\rho + 2k\rho_{\ell}) + k'(4\rho_{\ell})$?}
\end{lemma}
\begin{proof}
    By reversing the geodesics of $\mathcal{P}$ if necessary, without loss of generality we may assume that $\type{C} = \Brace{\forward{P_1},\forward{P_2},\dots,\forward{P_{k'}}} = \type{D}$.
	For $F\in \{C,D\}$, denote
    \begin{equation*}
        F\,=\,R^F_0,\,Q^F_1,\,R^F_1,\,Q^F_2,\,R^F_2,\,\ldots,\, R^F_{k'-1},\, Q^F_{k'},\, R^F_{k'}, \qquad\text{where}
    \end{equation*}
    \begin{itemize}
    \item $\Abs{R^C_0}=\Abs{R^D_0}=0$ and $\Abs{R^C_i},\Abs{R^D_i}\leq \rho_\ell$, for $i\in \{1,\ldots,k'\}$;
    \item $Q^C_i$ and $Q^D_i$ are subpaths of $P_i$, oriented the same way as $P_i$; and
    \item $\Abs{Q^C_i},\Abs{Q^D_i}> (2k+2)\rho_\ell$, for all $i\in \{1,\ldots,k'\}$.
    \end{itemize}
    We prove by induction over $k'$, that is, the number of geodesics the \snappaths $C$ and $D$ visit, that $\dist{}{\End{C},\End{D}} \leq \gamma+k'(4k\rho +2\rho_\ell+ 2k\rho_{\ell}) + k'(4\rho_{\ell})$.
    We conduct the induction on slightly relaxed objects --- \emph{relaxed $\ell$-simplified $\rho$-\snappaths} --- in which we allow the subpaths $Q^C_1$ and $Q^D_1$ to be arbitrarily short (in other words, the assertion $\Abs{Q^C_i},\Abs{Q^D_i}> (k+2)\rho_\ell$ only needs to hold for $i\geq 2$); clearly, this holds for the initial input.
    Let $\rho' \coloneqq \rho_{\ell}$ and $u \coloneqq \Start{C} = \Start{D}$. Note that $u\in \V{P_1}$ and $u=\Start{Q^C_1}=\Start{Q^D_1}$.

    In the case $k'=0$ we have that $u=\End{C}=\End{D}$.
    % \Me{No. We just have that $R^D_0$ and $R^C_0$ are the entire paths. They have bounded length but not length $0$.}
    Consider then the case $k' = 1$.
    %\Me{Don't we have to start at $k'=0$ actually?}
    %\Me{I think not, because we now consider relaxed snappaths. Should we write a comment of that?}
    Let $x \coloneqq \End{Q^C_1}$ and $y \coloneqq \End{Q^D_1}$.
    Without loss of generality, we assume that $x$ occurs on $P_1$ not later than $y$.
    %Denote $R_C \coloneqq R^C_1$ and $R_D \coloneqq R^D_1$ for brevity.
    By assumption, we have $\Abs{R^C_1},\Abs{R^D_1} \leq \rho'$.
    As we have $\gamma \geq \Abs{\Abs{D}-\Abs{C}} = \Abs{(\Abs{xP_1\,y}+\Abs{R^D_1})-\Abs{R^C_1}} \geq \Abs{xP_1\,y} - \rho'$, we obtain
    \begin{equation*}
        \dist{}{\End{C},\End{D}} \leq \Abs{R^C_1} + \Abs{xP_1\,y} + \Abs{R^D_1} \leq \gamma + 3\rho'.
    \end{equation*}

    We proceed to the main case $k' > 1$.
    Again, let $x \coloneqq \End{Q^C_1}=\Start{R^C_1}$ and $y \coloneqq \End{Q^D_1}=\Start{R^D_1}$ and assume, without loss of generality, that $x$ occurs on $P_1$ not later than $y$ does.
    Additionally, let $x' \coloneqq \End{R^C_1}=\Start{Q^C_2}$ and $y' \coloneqq \End{R^D_1}=\Start{Q^D_2}$; note that $x'$ and $y'$ lie on $P_2$.
    We distinguish three different cases: \textbf{(1)} $y'$ lies on $P_2$ before $x'$ and $z \coloneqq \End{Q^D_2}$ lies between them, \textbf{(2)} $y'$ lies on $P_2$ before $x'$ and $z$ lies after $x'$, and \textbf{(3)} $y'$ occurs on $P_2$ at the same point or after $x'$.
%    These cases are depicted in figures moved to \cref{app:figures}.

	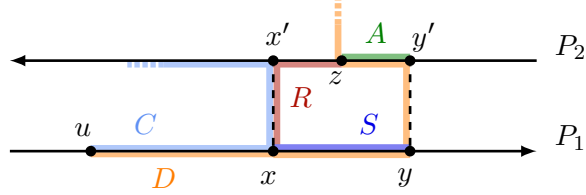
\begin{figure}[!ht]
		\centering
		%case 1 figure
		\begin{tikzpicture}[scale=1.2]
			\tikzstyle{geodesic} = [directededge]
			\tikzstyle{connector} = [edge,dashed]
			\colorlet{Pcolor}{myLightBlue!70!blue}
			\colorlet{Qcolor}{myOrange}
			\node (centre) at (0,0) {};
			\node (P-1-left) at ($(centre) + (-3,0)$) {};
			\node (P-1-right) at ($(centre) + (3,0)$) {};
			\node (P-2-left) at ($(P-1-left) + (0,1)$) {};
			\node (P-2-right) at ($(P-1-right) + (0,1)$) {};
			
			\node (P-1-label) at ($(P-1-right)+(30:0.3)$) {$P_1$};
			\node (P-2-label) at ($(P-2-right)+(30:0.3)$) {$P_2$};
			
			\node[vertex] (u) at ($(centre)+(-2,0)$) {};
			\node (u-label) at ($(u)+(110:0.25)$) {$u$};
			\node[vertex] (x) at ($(centre)+(0,0)$) {};
			\node (x-label) at ($(x)+(260:0.3)$) {$x$};
			\node[vertex] (y) at ($(x)+(1.5,0)$) {};
			\node (y-label) at ($(y)+(260:0.3)$) {$y$};
			\node[vertex] (x-p) at ($(x)+(0,1)$) {};
			\node (x-p-label) at ($(x-p)+(80:0.3)$) {$x'$};
			\node[vertex] (y-p) at ($(y)+(0,1)$) {};
			\node (y-p-label) at ($(y-p)+(60:0.3)$) {$y'$};
			\node[vertex] (z) at ($(x-p)!0.5!(y-p)$) {};
			\node (z-label) at ($(z)+(250:0.2)$) {$z$};
			
			\node (C-label) at ($(u)!0.3!(x)+(90:0.3)$) {\textcolor{Pcolor}{$C$}};
			\node (D-label) at ($(u)!0.4!(x)+(270:0.3)$) {\textcolor{Qcolor}{$D$}};
			
			\draw[geodesic] (P-1-left) to (P-1-right);
			\draw[geodesic] (P-2-right) to (P-2-left);
			\draw[connector] (x) to (x-p);
			\draw[connector] (y) to (y-p);
			
			\doublemarkedpath{u}{x}{180}{0}{Pcolor}{Qcolor}
			\markedpathbelow{x}{y}{180}{0}{Qcolor}
			\markedpathleft{y}{y-p}{270}{90}{Qcolor}
			\markedpathbelow{y-p}{z}{0}{180}{Qcolor}
			\markedpathleft{z}{$(z)+(0,0.4)$}{270}{90}{Qcolor}
			\markedpathleftDASH{$(z)+(0,0.4)$}{$(z)+(0,0.7)$}{270}{90}{Qcolor}
			\markedpathleft{x}{x-p}{270}{90}{Pcolor}
			\markedpathbelow{x-p}{$(x-p)+(-1.2,0)$}{0}{180}{Pcolor}
			\markedpathbelowDASH{$(x-p)+(-1.2,0)$}{$(x-p)+(-1.6,0)$}{0}{180}{Pcolor}
			
			% A
			\colorlet{Acolor}{myGreen}
			\markedpathabove{y-p}{z}{0}{180}{Acolor}
			\node (A-label) at ($(y-p)!0.5!(z)+(90:0.3)$) {\textcolor{Acolor}{$A$}};
			
			% R
			\colorlet{Rcolor}{myRed}
			\markedpathbelow{z}{x-p}{0}{180}{Rcolor}
			\markedpathright{x}{x-p}{270}{90}{Rcolor}
			\node (R-label) at ($(x-p)!0.4!(x)+(0:0.3)$) {\textcolor{Rcolor}{$R$}};
			
			% S
			\colorlet{Scolor}{myBlue}
			\markedpathabove{x}{y}{180}{0}{Scolor}
			\node (S-label) at ($(x)!0.7!(y)+(90:0.3)$) {\textcolor{Scolor}{$S$}};
		\end{tikzpicture}
		\caption{Illustration of case $\textbf{(1)}.$}
		\label{fig:case-1}
	\end{figure}

    In case \textbf{(1)}, see \cref{fig:case-1} as an illustration, we define $A \coloneqq y'D\, z = Q^D_2$, $R \coloneqq R^C_1 \cdot x'P_2 z$ and $S \coloneqq xP_1 y$.
    Because $y'P_2x'$ is a shortest path and $\Abs{R^C_1},\Abs{R^D_1}\leq \rho'$, we get $\Abs{R} - \rho' + \Abs{A} \leq \Abs{S} + 2\rho'$, and thus, $\Abs{R} \leq \Abs{S} + 3\rho' - \Abs{A}.$
    On the other hand, because $D$ is a $\delta$-almost shortest path, we obtain $\Abs{S} + \Abs{A} - \delta \leq \Abs{R}$.
    So we have
    \begin{align*}
        \Abs{S} + \Abs{A} - \delta \leq \Abs{S} + 3\rho' - \Abs{A},& \qquad \textrm{implying}\\
        \Abs{Q^D_2} = \Abs{A} \leq \frac{3\rho' + \delta}{2} \leq \frac{(2k+5)\rho' + 2k\rho}{2}\leq (2k+3)\rho'. &
    \end{align*}
    This contradicts the assumption that $D$ is a relaxed $\ell$-simplified $\rho$-\snappath. Hence, we conclude that this case simply cannot happen.

	\begin{figure}[!ht]
		\centering
        %case 2 figure
		\begin{tikzpicture}[scale=1.2]
			\tikzstyle{geodesic} = [directededge]
			\tikzstyle{connector} = [edge,dashed]
			\colorlet{Pcolor}{myLightBlue!70!blue}
			\colorlet{Qcolor}{myOrange}
			\colorlet{PPcolor}{myBlue}
			\colorlet{QQcolor}{myRed}
			\node (centre) at (0,0) {};
			\node (P-1-left) at ($(centre) + (-3,0)$) {};
			\node (P-1-right) at ($(centre) + (3,0)$) {};
			\node (P-2-left) at ($(P-1-left) + (0,1)$) {};
			\node (P-2-right) at ($(P-1-right) + (0,1)$) {};
			
			\node (P-1-label) at ($(P-1-right)+(30:0.3)$) {$P_1$};
			\node (P-2-label) at ($(P-2-right)+(30:0.3)$) {$P_2$};
			
			\node[vertex] (u) at ($(centre)+(-2,0)$) {};
			\node (u-label) at ($(u)+(110:0.25)$) {$u$};
			\node[vertex] (x) at ($(centre)+(0,0)$) {};
			\node (x-label) at ($(x)+(260:0.3)$) {$x$};
			\node[vertex] (y) at ($(x)+(1,0)$) {};
			\node (y-label) at ($(y)+(260:0.3)$) {$y$};
			\node[vertex] (x-p) at ($(x)+(0,1)$) {};
			\node (x-p-label) at ($(x-p)+(80:0.3)$) {$x'$};
			\node[vertex] (y-p) at ($(y)+(0,1)$) {};
			\node (y-p-label) at ($(y-p)+(60:0.3)$) {$y'$};
			
			\node (C-label) at ($(u)!0.3!(x)+(90:0.3)$) {\textcolor{Pcolor}{$C$}};
			\node (D-label) at ($(u)!0.4!(x)+(270:0.3)$) {\textcolor{Qcolor}{$D$}};
			
			\draw[geodesic] (P-1-left) to (P-1-right);
			\draw[geodesic] (P-2-right) to (P-2-left);
			\draw[connector] (x) to (x-p);
			\draw[connector] (y) to (y-p);
			
			\doublemarkedpath{u}{x}{180}{0}{Pcolor}{Qcolor}
			\markedpathbelow{x}{y}{180}{0}{Qcolor}
			\markedpathright{y}{y-p}{270}{90}{Qcolor}
			\markedpathleft{x}{x-p}{270}{90}{Pcolor}
			\markedpathabove{y-p}{$(x-p)+(-1,0)$}{0}{180}{Qcolor}
			\markedpathaboveDASH{$(x-p)+(-1,0)$}{$(x-p)+(-1.4,0)$}{0}{180}{Qcolor}
			\markedpathbelow{x-p}{$(x-p)+(-1.2,0)$}{0}{180}{Pcolor}
			\markedpathbelowDASH{$(x-p)+(-1.2,0)$}{$(x-p)+(-1.6,0)$}{0}{180}{Pcolor}
			
			% A
			\colorlet{Acolor}{myGreen}
			\markedpathabove{x}{y}{0}{180}{Acolor}
			\markedpathleft{y}{y-p}{270}{90}{Acolor}
			\markedpathbelow{x-p}{y-p}{0}{180}{Acolor}
			\node (A-label) at ($(y-p)!0.5!(y)+(180:0.3)$) {\textcolor{Acolor}{$A$}};
			
			% C'
			\draw[PPcolor,thin] ($(x-p)+(0,-0.07)$) to ($(x-p)+(-1.2,-0.07)$);
			\draw[PPcolor,thin,dashed,dash pattern=on 2pt off 1pt] ($(x-p)+(-1.2,-0.07)$) to ($(x-p)+(-1.6,-0.07)$);
			\node at ($(x-p)+(-1.6,-0.07)+(225:0.3)$) {\textcolor{PPcolor}{$C'$}};
			
			\draw[QQcolor,thin] ($(x-p)+(0,0.07)$) to ($(x-p)+(-1.2,0.07)$);
			\draw[QQcolor,thin,dashed,dash pattern=on 2pt off 1pt] ($(x-p)+(-1.2,0.07)$) to ($(x-p)+(-1.6,0.07)$);
			\node at ($(x-p)+(-1.6,0.07)+(135:0.3)$) {\textcolor{QQcolor}{$D'$}};
		\end{tikzpicture}
		\caption{Illustration of case $\textbf{(2)}.$}
		\label{fig:case-2}
	\end{figure}
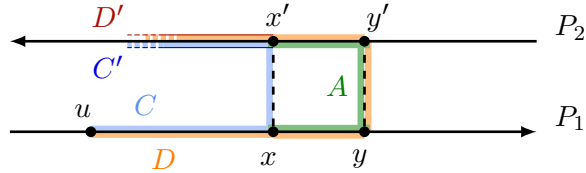

    We continue with case \textbf{(2)}.
    Here, we define $A \coloneqq xD\,x'$, see \cref{fig:case-2} for an illustration.
    As $Q$ is a $\delta$-almost shortest walk and $\Abs{R^C_1}=\Abs{xC\,x'}\leq \rho'$, we get $\Abs{A} \leq \rho' + \delta$.
    We define \snappaths:
    \begin{itemize}[nosep]
    \item $C' \coloneqq x'C$, i.e., $C'$ is the suffix of $C$ starting from $Q^C_2$; and 
    \item $D' \coloneqq x'D$, i.e., $D'$ is the suffix of $D$ starting from $Q^D_2$, but with $Q^D_2$ trimmed to $x'Q^D_2$.
    \end{itemize}
    Note that both $C'$ and $D'$ are again relaxed $\ell$-simplified $\rho$-\snappaths (here we make use of them being relaxed, as we shorten the subpaths of $D$ on $P_2$, which now is the first geodesic the walks visit) and visit $k' -1$ geodesics of $\mathcal{P}$.
    As $C'$ is a subwalk of $C$ and $D'$ is a subwalk of $D$, both are again $\delta$-almost shortest walks.
    Additionally, we obtain $\Abs{\Abs{C'}-\Abs{D'}} \leq \gamma + \Abs{\Abs{A}-\Abs{R^C_1}}\leq  \gamma+ \delta + 2\rho'$.
    Thus, by induction, we obtain that
    \begin{align*}
        \dist{}{\End{C},\End{D}} &\leq \gamma + \delta + 2\rho' + (k'-1)(4k\rho +2\rho'+ 2k\rho') +(k'-1)(4\rho') \\
        &\leq \gamma + k'(4k\rho + 2\rho'+ 2k\rho') + k'(4\rho').
    \end{align*}

	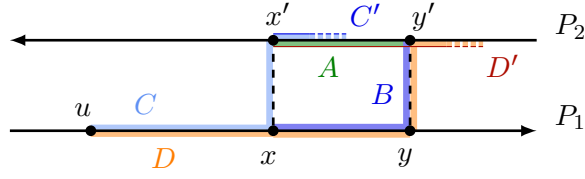
\begin{figure}[!ht]
		\centering
        %case 3 figure
		\begin{tikzpicture}[scale=1.2]
			\tikzstyle{geodesic} = [directededge]
			\tikzstyle{connector} = [edge,dashed]
			\colorlet{Pcolor}{myLightBlue!70!blue}
			\colorlet{Qcolor}{myOrange}
			\colorlet{PPcolor}{myBlue}
			\colorlet{QQcolor}{myRed}
			\node (centre) at (0,0) {};
			\node (P-1-left) at ($(centre) + (-3,0)$) {};
			\node (P-1-right) at ($(centre) + (3,0)$) {};
			\node (P-2-left) at ($(P-1-left) + (0,1)$) {};
			\node (P-2-right) at ($(P-1-right) + (0,1)$) {};
			
			\node (P-1-label) at ($(P-1-right)+(30:0.3)$) {$P_1$};
			\node (P-2-label) at ($(P-2-right)+(30:0.3)$) {$P_2$};
			
			\node[vertex] (u) at ($(centre)+(-2,0)$) {};
			\node (u-label) at ($(u)+(110:0.25)$) {$u$};
			\node[vertex] (x) at ($(centre)+(0,0)$) {};
			\node (x-label) at ($(x)+(260:0.3)$) {$x$};
			\node[vertex] (y) at ($(x)+(1.5,0)$) {};
			\node (y-label) at ($(y)+(260:0.3)$) {$y$};
			\node[vertex] (x-p) at ($(x)+(0,1)$) {};
			\node (x-p-label) at ($(x-p)+(80:0.3)$) {$x'$};
			\node[vertex] (y-p) at ($(y)+(0,1)$) {};
			\node (y-p-label) at ($(y-p)+(60:0.3)$) {$y'$};
			
			\node (C-label) at ($(u)!0.3!(x)+(90:0.3)$) {\textcolor{Pcolor}{$C$}};
			\node (D-label) at ($(u)!0.4!(x)+(270:0.3)$) {\textcolor{Qcolor}{$D$}};
			
			\draw[geodesic] (P-1-left) to (P-1-right);
			\draw[geodesic] (P-2-right) to (P-2-left);
			\draw[connector] (x) to (x-p);
			\draw[connector] (y) to (y-p);
			
			\doublemarkedpath{u}{x}{180}{0}{Pcolor}{Qcolor}
			\markedpathbelow{x}{y}{180}{0}{Qcolor}
			\markedpathright{y}{y-p}{270}{90}{Qcolor}
			\markedpathleft{x}{x-p}{270}{90}{Pcolor}
			\markedpathabove{x-p}{$(x-p)+(0.4,0)$}{180}{0}{Pcolor}
			\markedpathaboveDASH{$(x-p)+(0.4,0)$}{$(x-p)+(0.8,0)$}{180}{0}{Pcolor}
			\markedpathbelow{y-p}{$(y-p)+(0.4,0)$}{180}{0}{Qcolor}
			\markedpathbelowDASH{$(y-p)+(0.4,0)$}{$(y-p)+(0.8,0)$}{180}{0}{Qcolor}
			
			% A
			\colorlet{Acolor}{myGreen}
			\markedpathbelow{x-p}{y-p}{180}{0}{Acolor}
			\node (A-label) at ($(x-p)!0.4!(y-p)+(270:0.3)$) {\textcolor{Acolor}{$A$}};
			
			% B
			\colorlet{Bcolor}{myBlue}
			\markedpathabove{x}{y}{180}{0}{Bcolor}
			\markedpathleft{y}{y-p}{270}{90}{Bcolor}
			\node (B-label) at ($(y-p)!0.6!(y)+(180:0.3)$) {\textcolor{Bcolor}{$B$}};
			
			% C'
			\draw[PPcolor,thin] ($(x-p)+(0,0.07)$) to ($(x-p)+(0.4,0.07)$);
			\draw[PPcolor,thin,dashed,dash pattern=on 2pt off 1pt] ($(x-p)+(0.4,0.07)$) to ($(x-p)+(0.8,0.07)$);
			\node at ($(x-p)+(0.8,0.07)+(45:0.3)$) {\textcolor{PPcolor}{$C'$}};
			
			\draw[QQcolor,thin] ($(x-p)+(0,-0.07)$) to ($(y-p)+(0.4,-0.07)$);
			\draw[QQcolor,thin,dashed,dash pattern=on 2pt off 1pt] ($(y-p)+(0.4,-0.07)$) to ($(y-p)+(0.8,-0.07)$);
			\node at ($(y-p)+(0.8,-0.07)+(315:0.3)$) {\textcolor{QQcolor}{$D'$}};
		\end{tikzpicture}
		\caption{Illustration of case $\textbf{(3)}.$}
		\label{fig:case-3}
	\end{figure}

    Finally, we consider case \textbf{(3)}, where we define $A \coloneqq x'P_2 y'$ and $B \coloneqq xD\,y'$, see \cref{fig:case-3} for an illustration.
    In this case we define \snappaths:
    \begin{itemize}
        \item $C' \coloneqq x'C$, i.e., $C'$ is the suffix of $C$ starting from $Q^C_2$; and
        \item $D' \coloneqq A \cdot y'D$, i.e., $D'$ is the suffix of $D$ starting from $Q^D_2$, but with $Q^D_2$ extended by prepending~$A$.
    \end{itemize}
    Note that thus, both $C'$ and $D'$ are relaxed $\ell$-simplified $\rho$-\snappaths that start at $x'$ and visit $k'-1$ geodesics of $\Pp$.
    As $P_2$ is a geodesic, we know that $A$ is a shortest path and hence $\Abs{A} \leq \Abs{B} + \Abs{R^C_1}\leq \Abs{B}+\rho'$. On the other hand, $xBy$ is a shortest path as well, hence $\Abs{B}=\Abs{xBy}+\Abs{R^D_1}\leq \Abs{R^C_1}+\Abs{A}+2\Abs{R^D_1}\leq \Abs{A}+3\rho'$.
    Thus $\Abs{\Abs{A}-\Abs{B}} \leq 3\rho'$, and
    \begin{align*}
        \Abs{(\Abs{uP_1x} + \Abs{R^C_1} + \Abs{C'}) - (\Abs{uP_1x} + \Abs{B} + \Abs{D'} - \Abs{A})} \leq \gamma,& \qquad\textrm{implying}\\
        \Abs{\Abs{C'} - \Abs{D'}} \leq \gamma + 4\rho'.&
    \end{align*}
    As $xD$ is $\delta$-almost shortest, by $\Abs{A}\leq \Abs{B}+\rho'$ we obtain that $R^C_1 \cdot D'$ is $(\delta + 2\rho')$-almost shortest, which in turn implies that $D'$ is $(\delta + 2\rho')$-almost shortest as well.
    $C'$ as a subpath of $C$ is again $\delta$-almost shortest.
    Note that $(\delta + 2\rho') \leq 4k\rho + 2\rho'+(k-(k'-1))2\rho'$.
    %Next, observe that $\Abs{xB\,y} \leq \Abs{A} + 2\rho'$, for $xB\,y$ is a geodesic and $\Abs{R^C_1},\Abs{R^D_1} \leq \rho'$.
    So we may apply induction to $C'$ and $D'$ to  obtain~that
    \begin{align*}
        \dist{}{\End{C},\End{D}} &\leq \gamma + 4\rho' + (k'-1)(4k\rho +2\rho'+ 2k\rho') +(k'-1)(4\rho') \\
        &\leq \gamma + k'(4k\rho +2\rho'+ 2k\rho') + k'(4\rho').
    \end{align*}

    Note that, indeed, none of the cases uses any lower bound on $\Abs{Q^C_1}$ or $\Abs{Q^D_1}$, so relaxed $\ell$-simplified $\rho$-\snappaths suffice.
\end{proof}

We have the following easy consequence that concerns \snappaths with the same start vertex, but not necessarily lying on a geodesic from $\Pp$.

\begin{restatable}{corollary}{notOnGeodesic}
    \label{corr:same_type_impl_near_end}
	Let $G$ be a graph with a $(k,\rho)$-geodesic-cover $\mathcal{P}$. Suppose $C$ and $D$ are two $\ell$-simplified $\rho$-\snappaths with
	\begin{itemize}
		\item $\Start{C}=\Start{D}$,
		\item $\type{C} = \type{D}$,
		\item $\Abs{\Abs{C} - \Abs{D}} \leq \gamma$ for some non-negative integer $\gamma$, and
        \item $C,D$ are $\delta$-almost shortest paths, where $\delta = 4k\rho$.
	\end{itemize}
	Then $\dist{}{\End{C},\End{D}} \leq \gamma+k(4k\rho + 2\rho_\ell +2k\rho_{\ell})+(6k+4)\rho_{\ell}$.
	% Then, $\dist{}{\End{P},\End{Q}} \leq \beta+\gamma + 4(k+1)\rho$.
\end{restatable}
\begin{proof}
	Consider first the corner case $k'=0$. Then both $C$ and $D$ have length at most $\rho'$, hence by triangle inequality we have  $\dist{}{\End{C},\End{D}}\leq 2\rho_\ell$.
	
	So from now on assume that $k'\geq 1$.
	As $C$ and $D$ have the same type, there is a $P_1 \in \mathcal{P}$ such that the first two walks on $C$ and $D$, call them $R^C_0$ and $R^D_0$, have their ends on $P_1$. We also have $\Abs{R^C_0},\Abs{R^D_0}\leq \rho_\ell$.
	We may assume without loss of generality that $C$ and $D$ follow $P_1$ in the direction of $P_1$, and that $x \coloneqq \End{R^C_0}$ lies on $P_1$ not before $y \coloneqq \End{R^D_0}$.
	
	We define two new \snappaths:
	\begin{itemize}
		\item $C' \coloneqq y P_1 x C$, that is, $C'$ is the suffix of $C$ after $R^C_0$, but with the path $yP_1x$ prepended to the walk after $R^C_0$.
		\item $D'\coloneqq yD$, that is, $D'$ is the suffix of $D$ after $R^D_0$.
	\end{itemize}
	Clearly, both $C'$ and $D'$ are $\ell$-simplified \snappaths that visit $k'$ geodesics of $\Pp$. Note since $P_1$ is a geodesic, we have
	\[\dist{}{x,y}=\Abs{xP_1\,y}\leq \Abs{R^C_0}+\Abs{R^D_0}\leq 2\rho_\ell.\]
	In particular, both $xC$ and $xD$ are $\delta$-almost shortest walks due to being suffixes of $\delta$-almost shortest walks, hence $C'$ and $D'$ are both $(\delta+2\rho_\ell)$-almost shortest walks.
	Further, observe that
	\begin{align*}
		\Abs{\Abs{C'} - \Abs{D'}} ={} & \Abs{(\Abs{C}-\Abs{R^C_0}+\Abs{yP_1x}) - (\Abs{D}+\Abs{R^D_0})}\\
		={} & \Abs{\Abs{C}-\Abs{D}-\Abs{R^C_0}+\Abs{R^D_0}+\dist{}{x,y}}\\
		\leq{} &  \Abs{\Abs{C}-\Abs{D}}+2\Abs{R^C_0}+2\Abs{R^D_0}\leq \gamma+4\rho_{\ell}.
	\end{align*}
	Thus, we may apply \cref{lem:same_start_and_same_type_impl_near_end} to $C'$ and $D'$. From this we conclude that
	\begin{align*}\dist{}{\End{C'},\End{D'}} \leq{} & (\gamma + 4\rho_{\ell}) + k(4k\rho +2\rho_\ell+ 2k\rho_{\ell}) +k(4\rho_{\ell})\\
		={} &\gamma+k(4k\rho + 2k\rho_{\ell})+(6k+4)\rho_{\ell}.\end{align*}
	This finishes the proof as $\End{C} = \End{C'}$ and $\End{D} = \End{D'}$.
\end{proof}

In the next lemma, we use the tools developed so far to show that every BFS layer is well-coverable.

\begin{lemma}\label{lem:sphereCoverable}
	Let $G$ be a $(k,\rho)$-geodesic-coverable graph.
	Then for every $u \in \V{G}$ and every non-negative integer $d$, the set
			$\sphere{d}{u} \coloneqq \Set{v \in \V{G} \mid \dist{}{u,v} = d}$
	is $\Brace{260\cdot 14^k\cdot k^{2k+4},2\rho}$-coverable.
\end{lemma}
\begin{proof}
    For every $v \in \sphere{d}{u}$ we fix a shortest $u$-$v$-path $P_v$ of length $d$.
    We also fix a $\rho$-\snappath $Q^\circ_v$ corresponding to $P_v$ in the sense of \cref{lem:every_path_can_be_snapped}. Note that  we have $d\leq \Abs{Q^\circ_v}\leq d+4k\rho$. We may then apply \cref{lem:every_snal_path_can_be_simplified} to find an $\ell$-simplified \snappath $Q_v$ with $\concat{Q_v}=\concat{Q^\circ_v}$, for some $\ell\in \Set{0,1,\ldots,k}$. Note that $Q_v$ is a $4k\rho$-almost shortest path.
    Let $\mathcal{Q} \coloneqq \Set{Q_v \colon v \in \sphere{d}{u}}$ and $\mathcal{T} \coloneqq \Set{\type{Q} \colon Q \in \mathcal{Q}}.$

    For every type $t \in \mathcal{T}$, we now fix an endvertex $x_t$ of some path $Q \in \mathcal{Q}$ with $\type{Q} = t$.
    By \cref{corr:same_type_impl_near_end,lem:every_snal_path_can_be_simplified}, $\type{Q_v} = t$ entails that $\End{P_v}=\End{Q_v} \in \ball{k(4k\rho + 2k\rho_{k})+(6k+4)\rho_{k}}{x_t}.$
    Thus, $\sphere{d}{u}$ is coverable by $\Abs{\mathcal{T}} \leq (2k)^{k+1}$ balls of radius $k(4k\rho + 2k\rho_{k})+(6k+4)\rho_{k}\leq 16k^2\rho_k\leq 48k^2(2k+5)^k\cdot \rho$.
    
    Now, we infer from  \cref{lem:balls-coverable}, that every such ball is again $(2k(48k^2(2k+5)^k+1),2\rho)$-coverable, so also $(100k^3(7k)^k,2\rho)$-coverable.
    Thus, $\sphere{d}{u}$ is $\Brace{(2k)^{k+1}\cdot 100k^3(7k)^k,2\rho}$-coverable, hence also $\Brace{200\cdot 14^k\cdot k^{2k+4},2\rho}$-coverable.
\end{proof}

We can now prove our main result, which we restate for convenience.
Note that in the adopted terminology, the properties expected from $C_1,\ldots, C_\ell$ just assert that it is a distance-$2\rho$ path-partition-decomposition of width $k^{\Oh{k}}$.

\thmMain*
\begin{proof}
    Write $\eta\coloneqq 200\cdot 14^k\cdot k^{2k+4}$ for brevity, and note that $\eta\in k^{\Oh{k}}$.
	Without loss of generality, we may assume that $G$ is connected; otherwise, we may apply the reasoning to every connected component of $G$ separately and naturally combine the outcomes. Fix an arbitrary vertex $u$ of $G$ and for $i=1,2,3,\ldots$,~define
    \begin{equation*}
        C_i\coloneqq \Set{v\in V(G)\mid (i-1)\cdot 2\rho \leq \dist{}{u,v}<i\cdot 2\rho}.
    \end{equation*}
    Letting $\ell$ be the largest integer such that $C_\ell$ is non-empty, we see that $C_1,\ldots, C_\ell$ is a distance-$2\rho$ path-partition-decomposition of $G$. Further, for each $i\in \Set{1,\ldots,\ell}$ we have, by \cref{lem:sphereCoverable}, that $\sphere{(i-1)\cdot 2\rho}{u}$ is $\Brace{\eta,2\rho}$-coverable. Since every vertex of $C_i$ is at distance at most $2\rho$ from some vertex of $\sphere{(i-1)\cdot 2\rho}{u}$, we conclude that $C_i$ is $\Brace{\eta,4\rho}$-coverable, hence also $\Brace{10k\eta,2\rho}$-coverable by \cref{lem:balls-coverable}. So $C_1,\ldots,C_\ell$ is a distance-$2\rho$ path-partition-decomposition of $G$ of width at most $10k\eta\in k^{\Oh{k}}$, as required.

    To construct the quasi-isometry, let $I$ be a maximal distance-$4\rho$ independent set in $G$ and let $H\coloneqq \DistGraph{4\rho}{I}{G}$ be the $(I,4\rho)$-distance graph of $G$.
    Further, let $\hat{H}\coloneqq \whDistGraph{4\rho}{I}{G}$ be the graph obtained from $H$ by replacing every edge by a path of length $12\rho$.
    By \cref{thm:distance-graph}, $G$ is $(3,12\rho)$-quasi-isometric to $\hat{H}$, hence it remains to prove that $\pw{\hat{H}}\leq k^{\Oh{k}}$ and $\Delta(\hat{H})\leq 26k$. Since subdividing edges does not change the pathwidth and the maximum degree (assuming they are at least $2$), it suffices to show that (i) $\pw{H}\leq k^{\Oh{k}}$ and (ii) $\Delta(H)<26k$.
    
    For the first property, for $i\in \Set{1,\ldots,\ell}$ we set 
        $D_i\coloneqq (C_i\cup C_{i+1}\cup C_{i+2}\cup C_{i+3})\cap I.$
    Recall that each set $C_i$ is $\Brace{10k\eta,2\rho}$-coverable, hence each set $D_i$ is $\Brace{40k\eta,2\rho}$-coverable. As $D_i$ is also a distance-$4\rho$ independent set, from \cref{lem:is-ds-comparison} we conclude that $\Abs{D_i}\leq 40k\eta\in k^{\Oh{k}}$. It remains to argue that $D_1,\ldots,D_\ell$ is a path decomposition of $H$. The only non-trivial check is that whenever vertices $u,v\in I$ are adjacent in $H$, there is some $t$ such that $u,v\in D_t$. Recall that $u,v$ being adjacent in $H$ entails $\dist{G}{u,v}\leq 12\rho$. Hence, if $i$ and $j$ are such that $u\in C_i$ and $v\in C_j$, then $|i-j|\leq 3$. It follows that $u,v\in D_t$ where $t=\min(i,j)$.

    Finally, for the second property, consider any  $u\in I$. By the definition of the distance graph, for every neighbour $v$ of $u$ in $H$ we have $\dist{G}{u,v}\leq 12\rho$. So $\Set{v\in I \mid \dist{H}{u,v}\leq 1}$ is a distance-$4\rho$ independent set contained in $\ball{12\rho}{v}$. By \cref{lem:balls-coverable}, $\ball{12\rho}{v}$ is $(26k,2\rho)$-coverable. From \cref{lem:is-ds-comparison} we now infer that
    $|\Set{v\in I \mid \dist{H}{u,v}\leq 1}|\leq 26k$,
    so the degree of $u$ in $H$ is smaller than~$26k$.
\end{proof}

\section{Algorithms}\label{sec:algorithms}

In this \namecref{sec:algorithms}, we present two algorithms on graphs of bounded distance tree-partition width.
The first one finds the largest size of a distance-$2\rho$ independent set; the second finds the smallest size of a distance-$\rho$ dominating set.

\thmDistIS*
\begin{proof}
    Let $\Tt=(T,\Set{C_x\colon x\in \V{T}})$ be the provided tree-partition-decomposition of $G$. For a node $x\in \V{T}$ and a vertex set $S\subseteq \V{G}$, we denote
    \[S[x]\coloneqq \Set{u\in S \mid \dist{}{u,C_x}\leq \rho}.\]
    We first observe that for any distance-$2\rho$ independent set $I$, the sets $I[x]$ are all~small.

    \begin{claim}\label{cl:is-small-vicinity}
       For every distance-$2\rho$ independent set $I$ in $G$, we have
       \begin{equation*}
       		|I[x]|\leq (\Delta+1)k\qquad\textrm{for all }x\in \V{T}.
       	\end{equation*}
    \end{claim}
    \begin{claimproof}
        Since $\Tt$ is a distance-$\rho$ tree-partition-decomposition of $G$, we have
        \begin{equation*}
        	I[x]\subseteq \bigcup \Set{C_y\mid \dist{T}{x,y}\leq 1}.
        \end{equation*}
        By assumption, there are at most $\Delta+1$ parts $C_y$ considered on the right-hand side above, and each of them is $(k,\rho)$-coverable; so their union is $(k(\Delta+1),\rho)$-coverable. Hence, $I[x]$ is a distance-$2\rho$ independent set that is $(k(\Delta+1),\rho)$-coverable, so $|I[x]|\leq k(\Delta+1)$ by \cref{lem:is-ds-comparison}.    
    \end{claimproof}

    The idea of the algorithm is to construct a maximum independent set $I$ by a bottom-up dynamic programming over $T$, where the states at node $x$ correspond to candidates for the sets $I[x]$. \cref{cl:is-small-vicinity} ensures that we may only consider candidates of size at most $k(\Delta+1)$. We proceed to formal~details.

    Let us root $T$ in an arbitrary node $r$; this imposes the parent-child and ancestor-descendant relations in $T$. For a node $x\in \V{T}$, we denote
    \[V_x\coloneqq C_x\cup \bigcup \{C_y\colon y\textrm{ is a descendant of }x\textrm{ in }T\}\]
    and
    \begin{align*}
    	\Ii_x\coloneqq \{A\subseteq \V{G}\mid A\textrm{ is distance-}2\rho\textrm{ independent and\phantom{$\}.$}}\\
    	\textrm{distance-}\rho\textrm{ dominated by }C_x\}.
    \end{align*}
    By \cref{cl:is-small-vicinity}, every member of $\Ii_x$ has cardinality at most $k(\Delta+1)$, hence we may construct all the families $\Ii_x$ in total time $n^{\bigO{\Delta k}}$. For every $A\in \Ii_x$, we define the following quantity:
    \begin{align*}
    	\Fkt{\Phi_x}{A}\coloneqq \max \{|W|\colon W\subseteq V_x, \dist{}{W,C_x}>\rho,\textrm{ and \phantom{$\}\}\}\}\}$.}}\\
    	A\cup W\textrm{ is distance-}2\rho\textrm{ independent}\}.
    \end{align*}
    Sets $W$ over which the maximum ranges in the definition above will be called {\em{extensions}} of $A$.
    Clearly, we have
    \begin{equation}\label{eq:turtle}
    	\distIS{2\rho}{G}=\max \Set{|A|+\Fkt{\Phi_r}{A}\colon A\in \Ii_r},
    \end{equation}
    and $\Fkt{\Phi_x}{A}=0$ for all $A\in \Ii_x$ whenever $x$ is a leaf of $T$. It remains to show how to compute the values $\Fkt{\Phi_x}{\cdot}$ for non-leaf nodes of $T$ in a bottom-up manner. For this, consider the following definition: for two nodes $x,y$ of $T$, we say that $A\in \Ii_x$ and $B\in \Ii_y$ are \emph{compatible} if
    \[A\cup B\textrm{ is distance-}2\rho\textrm{ independent}\qquad\textrm{and}\qquad B[x]=A\cap B=A[y].\]

    \begin{claim}\label{cl:transitionIS}
        For every $x\in \V{T}$ and $A\in \Ii_x$, we have
        \[\Fkt{\Phi_x}{A}=\sum_{y\colon \textrm{child of }x} \max \Set{|B\setminus A|+\Fkt{\Phi_y}{B}\colon B\in \Ii_y\textrm{ and }A\textrm{ and }B\textrm{ are compatible}}.\]
    \end{claim}
    \begin{claimproof}
        Let $M$ be the value of the right-hand side of the claimed equality.
        First, let $W^\star$ be a maximum-size extension of $A$ and let $I^\star \coloneqq A\cup W^\star$. For each child $y$ of $X$, let
        \[B_y\coloneqq I^\star[y]\qquad\textrm{and}\qquad W_y\coloneqq (I^\star\cap V_y)\setminus B_y.\]
        It is straightforward to verify from the definitions that
        \begin{itemize}[nosep]
            \item $\Set{B_y\setminus A,W_y\colon y\textrm{ is a child of }x}$ is a partition of $W_x$; and
            \item for each child $y$ of $x$, $B_y\in \Ii_y$, $A$ and $B_y$ are compatible, and $W_y$ is an extension of $B_y$. 
        \end{itemize}
        Therefore, we have 
        \[\Fkt{\Phi_x}{A}=|W^\star|=\sum_{y\colon \textrm{child of }x} |B_y\setminus A|+|W_y|\leq M.\]
        This establishes the inequality in one direction.

        For the other direction, for every child $y$ of $x$ let us fix a set $B_y\in \Ii_y$ compatible with $A$ and an extension $W_y$ of $B_y$ that maximize $|B_y\setminus A|+|W_y|$. Thus $M=\sum_{y\colon \textrm{child of }x} |B_y\setminus A|+|W_y|$. To show that $\Fkt{\Phi_x}{A}\leq M$, it suffices to prove~that
        \[W\coloneqq \bigcup_{y\colon \textrm{child of }x} (B_y\setminus A)\cup W_y\]
        is an extension of $A$. That $W\subseteq V_x$ and $\dist{}{W,C_x}>\rho$ follows directly from the compatibility of $A$ and each $B_y$ and the assumption that $\Tt$ is a distance-$\rho$ tree-partition-decomposition. So it remains to argue that $A\cup W$ is a distance-$2\rho$ independent~set.

        Consider any distinct $u,v\in A\cup W$. If $u,v\in A\cup B_y$ for some child $y$ of $x$, then $\dist{}{u,v}> 2\rho$ because $A\cup B_y$ is distance-$2\rho$ independent (by compatibility). Further, if $u,v\in B_y\cup W_y$ for some child $y$ of $x$, then again $\dist{}{u,v}> 2\rho$ as $B_y\cup W_y$ is distance-$2\rho$ independent. We are left with two~cases.

        First, suppose that for some child $y$ of $x$, we have $u\in A\setminus B_y$ and $v\in W_y$. As $W_y$ is an extension of $B_y$, we have $v\in W_y\subseteq V_y\setminus C_y$ and $\dist{}{v,C_y}>\rho$. In particular, the node $y'$ of $T$ such that $v\in C_{y'}$ is a strict descendant of $y$. On the other hand, since $u\in A\setminus B_y$ and $A$ and $B_y$ are compatible, we have $\dist{}{u,C_y}>\rho$ and the node $x'$ of $T$ such that $u\in C_{x'}$ much be either equal to $x$ or adjacent to $x$ in $T$. In particular, the path in $T$ connecting $x'$ and $y'$ must pass through~$y$, so every path in $G$ connecting $u$ with $v$ must pass through $C_y$. Let then $P$ be a shortest $u$-$v$-path and let $w$ be a vertex of $C_y$ on $P$. Since $\dist{}{u,C_y}>\rho$, the prefix of $P$ between $u$ and $w$ must be of length larger than~$\rho$. Similarly, since $\dist{}{v,C_y}>\rho$, the suffix of $P$ between $w$ and $v$ must be of length larger than~$\rho$. We conclude that the length of $P$ is larger than $2\rho$, so~$\dist{}{u,v}>2\rho$.

        We are left with the following case: for some distinct children $y,z$ of $x$, we have $u\in (B_y\setminus A)\cup W_y$ and $v\in (B_z\setminus A)\cup W_z$. By the compatibility of $A$ and $B_y$, we have $\dist{}{u,C_x}>\rho$ and the node $y'$ such that $u\in C_{y'}$ is either equal to $y$ or is a descendant of $y$ in $T$. Similarly, $\dist{}{v,C_x}>\rho$ and the node $z'$ such that $v\in C_{z'}$ is either equal to $z$ or is a descendant of $z$. Therefore, the $y'$-$z'$-path in $T$ passes through $x$, so every $u$-$v$-path in $G$ must pass through~$C_x$. In particular, if $P$ is a shortest $u$-$v$-path, then $P$ must contain a vertex $w$ belonging to $C_x$. Then the $u$-$w$-prefix of $P$ must be of length larger than~$\rho$, due to $\dist{}{u,C_x}>\rho$, and similarly the $w$-$v$-suffix of $P$ must be of length larger than $\rho$, due to $\dist{}{v,C_x}>\rho$. Hence the length of $P$ is larger than~$2\rho$, witnessing that $\dist{}{u,v}>2\rho$.
    \end{claimproof}

    The algorithm now proceeds as follows: compute all the values $\Fkt{\Phi_x}{A}$ for all $x\in \V{T}$ and $A\in \Ii_x$ in a bottom-up manner using the formula of \cref{cl:transitionIS}, and output the value of $\distIS{2\rho}{G}$ as prescribed by~\eqref{eq:turtle}. There are $n^{\Oh{\Delta k}}$ values to compute and computing each of them takes time  $n^{\Oh{\Delta k}}$, hence the overall time complexity is $n^{\Oh{\Delta k}}$.
\end{proof}

\thmDistDS*
\begin{proof}
    Let $\Tt=(T,\Set{C_x\colon x\in \V{T}})$ be the given tree-partition-decomposition of~$G$. We use the same notation as in the proof of \cref{thm:distIS_xp}: for a set $S\subseteq \V{G}$ and a node $x\in \V{T}$, $S[x]$ denotes the subset of $S$ comprising all vertices of $S$ at distance at most $\rho$ from $C_x$. We have the following analogue of \cref{cl:is-small-vicinity}; note that it applies only to dominating sets of minimum size.

 \begin{claim}\label{cl:ds-small-vicinity}
       For every minimum-size distance-$\rho$ dominating set $D$ in $G$, we have
       \[|D[x]|\leq (\Delta^2+1)k\qquad\textrm{for all }x\in \V{T}.\]
    \end{claim}
    \begin{claimproof}
        Since $\Tt$ is a distance-$\rho$ tree-partition-decomposition, we have \[D[x]\subseteq \bigcup \Set{C_y\colon y\in \V{T}, \dist{T}{x,y}\leq 1}.\] Further, any vertex that is distance-$\rho$ dominated by a vertex of $D[x]$ belongs to \[R\coloneqq \bigcup \Set{C_y\colon y\in \V{T}, \dist{T}{x,y}\leq 2}.\] Note that $R$ is the union of at most $\Delta(\Delta-1)+\Delta+1=\Delta^2+1$ parts $C_y$, and each of those parts is $(k,\rho)$-coverable by assumption. So $R$ is $(k(\Delta^2+1),\rho)$-coverable, meaning that there is $F\subseteq \V{G}$ with $|F|\leq k(\Delta^2+1)$ that distance-$\rho$ dominates~$R$. In particular, $D'\coloneqq (D\setminus D[x])\cup F$ is a distance-$\rho$ dominating set in $G$. By the minimality of $D$ we have $|D|\leq |D'|$, implying that $|D[x]|\leq |F|\leq k(\Delta^2+1)$.
    \end{claimproof}
    Again, we root $T$ in arbitrary node $r$ and for a node $x$, by $V_x$ we denote the union of the part $C_x$ and all the parts $C_y$ for $y$ ranging over the descendants of~$x$. We also define
    \begin{align*}
    	\Dd_x\coloneqq \{A\subseteq \V{G}\mid |A|\leq k (\Delta^2+1), A\textrm{ distance-}\rho\textrm{ dominates $C_x$, and\phantom{$\}\}.$}}\\
    	A\textrm{ is distance-}\rho\textrm{ dominated by }C_x\}.
    \end{align*}
    Again, for every node $x$ we have that $|\Dd_x|\leq n^{\Oh{\Delta^2 k}}$ and all the sets $\Dd_x$ can be computed in time~$n^{\Oh{\Delta^2 k}}$.
    
    Call a set $S\subseteq \V{G}$ \emph{thin} if $|S[x]|\leq k(\Delta^2+1)$ for each node $x$ of $T$. With these definitions, we can define the values we are going to compute by dynamic programming. For every node $x$ of $T$ and $A\in \Dd_x$, we define
    \begin{align*}
    	\Fkt{\Psi_x}{A}\coloneqq \min\{|W|\colon W\subseteq V_x,\dist{}{W,C_x}>\rho, A\cup W\textrm{ is thin, and \phantom{$\}\}\}\}\}\}\}\}\}\}\}.$.}}\\ A\cup W\textrm{ distance-}\rho\textrm{ dominates }V_x\}.
    \end{align*}
    As in the proof of \cref{thm:distIS_xp}, sets $W$ as in the formula above will be called \emph{extensions} of $A$.
    By \cref{cl:ds-small-vicinity}, we have 
    \begin{equation}\label{eq:beaver}
    	\distDS{\rho}{G} = \min \Set{|A|+\Fkt{\Psi_r}{A}\colon A\in \Dd_r},
    \end{equation}
    and $\Fkt{\Psi_x}{A}=0$ for all $A\in \Ii_x$ whenever $x$ is a leaf of $T$. So it remains to prove an analogue of \cref{cl:transitionIS}: how to compute the values $\Fkt{\Psi_x}{\cdot}$ by bottom-up dynamic programming. For this, we will use the following notion of compatibility: for distinct nodes $x,y\in \V{T}$, call $A\in \Dd_x$ and $B\in \Dd_y$ {\em{compatible}}~if
    \[B[x]=A\cap B=A[y].\]

    \begin{claim}\label{cl:transitionDS}
        For every $x\in \V{T}$ and $A\in \Dd_x$, we have
        \[\Fkt{\Psi_x}{A}=\sum_{y\colon \textrm{child of }x} \min \Set{|B\setminus A|+\Fkt{\Psi_y}{B}\colon B\in \Dd_y\textrm{ and }A\textrm{ and }B\textrm{ are compatible}}.\]
    \end{claim}
    \begin{claimproof}
        Again, by $M$, we denote the right-hand side of the formula. We prove the equality $\Fkt{\Psi_x}{A}=M$ by arguing the two inequalities.

        For the inequality $\Fkt{\Psi_x}{A}\geq M$, let $W^\star$ be a minimum-size extension of $A$; so $|W^\star|=\Fkt{\Psi_x}{A}$. Also, let $I^\star\coloneqq A\cup W^\star$. For each child $y$ of $x$, define
        \[B_y\coloneqq I^\star[y]\qquad\textrm{and}\qquad W_y\coloneqq (I^\star\cap V_y)\setminus B_y.\]
        It follows immediately from the definitions and the assumption that $\Tt$ is a distance-$\rho$ tree-partition-decomposition of $G$ that
        \begin{itemize}%[nosep]
            \item $\Set{B_y\setminus A,W_y\colon y\textrm{ is a child of }x}$ is a partition of $W_x$; and
            \item for each child $y$ of $x$, $B_y\in \Dd_y$ and $A$ and $B_y$ are compatible. 
        \end{itemize}
        We are left with verifying that for each child $y$ or $x$, $W_y$ is an extension of $B_y$. Note that if we argue this, then we have
        \[\Fkt{\Psi_x}{A}=|W^\star|=\sum_{y\colon \textrm{child of }x} |B_y\setminus A|+|W_y|\geq M,\]
        as required.

        Consider  any child $y$ of $x$. From the definitions it follows that $W_y\subseteq V_y$ and $\dist{}{W_y,C_y}>\rho$, and the thinness of $B_y\cup W_y$ follows from the thinness of $A\cup W$. We only need to argue that $B_y\cup W_y$ distance-$\rho$ dominates $V_y$. First, note that since $I^\star$ distance-$\rho$ dominates $V_x$, we have that $B_y=I^\star[y]$ distance-$\rho$ dominates $C_y$. Similarly, for every strict descendant $y'$ of $y$, the set $I^\star[y']$ distance-$\rho$ dominates $C_{y'}$. Since $\Tt$ is a distance-$\rho$ tree-partition-decomposition, we have $I^\star[y']\subseteq \bigcup_{y''\colon \dist{T}{y'}{y''}\leq 1} C_{y''}$, which in particular implies that $I^\star[y']\subseteq I^\star\cap V_y\subseteq B_y\cup W_y$. All in all, $B_y\cup W_y$ distance-$\rho$ dominates the whole~$V_y$, as required. This settles the inequality $\Fkt{\Psi_x}{A}\geq M$.
   
        For the converse inequality $\Fkt{\Psi_x}{A}\leq M$, for every child $y$ of $x$ we fix a set $B_y\in \Dd_y$ compatible with $A$ and an extension $W_y$ of $B_y$ that minimize $|B_y\setminus A|+|W_y|$; so $M=\sum_{y\colon \textrm{child of }x} |B_y\setminus A|+|W_y|$. We now observe~that
        \[W\coloneqq \bigcup_{y \colon \textrm{child of }x} (B_y\setminus A)\cup W_y\]
        is an extension of $A$; this proves that $\Fkt{\Psi_x}{A}\leq M$.
        Indeed, that $W\subseteq V_x$, $\dist{}{W,C_x}>\rho$, and $A\cup W$ is thin follows directly from the compatibility of $A$ and each $B_y$ and the thinness of each set $B_y\cup W_y$. Further, $A$ distance-$\rho$ dominates $C_x$ while $B_y\cup W_y$ distance-$\rho$ dominates  $V_y$, for each child $y$ of $x$. Hence, $A\cup W=A\cup \bigcup_{y\colon \textrm{child of }x} (B_y\cup W_y)$ together distance-$\rho$ dominates the whole~$V_x$. So $W$ is indeed an extension of $A$ and we are done.
    \end{claimproof}

As in the proof of \cref{thm:distIS_xp}, the algorithm proceeds as follows: compute the values $\Fkt{\Psi_x}{A}$ for all $x\in \V{T}$ and $A\in \Dd_x$ in a bottom-up manner, and output the value of $\distDS{\rho}{G}$ as stipulated by~\eqref{eq:beaver}. There are $n^{\Oh{\Delta^2 k}}$ values to compute and computing each of them takes time  $n^{\Oh{\Delta^2 k}}$, hence the overall time complexity is $n^{\Oh{\Delta^2 k}}$.
\end{proof}

\bibliographystyle{alphaurl}
\bibliography{literature}

%\newpage
%\appendix
%
%\input{proofs}
%
%\input{figures}
%

\end{document}